\documentclass[12pt,twoside]{article}

\date{}

\usepackage{mathrsfs}
\usepackage{amsmath}
\usepackage{amssymb}
\usepackage[hidelinks]{hyperref}
\usepackage{longtable}
\usepackage{amsthm}

\begin{document}

\centerline{\Large{\bf Groups of order 64 are determined}} 

\centerline{} 

\centerline{\Large{\bf by their Tables of Marks}} 

\centerline{} 

\centerline{\bf {Peter Bonart}} 

\centerline{} 

\centerline{University of Regensburg} 

\centerline{Faculty of Mathematics}

\centerline{e-mail: peterjbonart@gmail.com}

\newtheorem{Theorem}{\quad Theorem}[section] 

\newtheorem{Definition}[Theorem]{\quad Definition} 

\newtheorem{Corollary}[Theorem]{\quad Corollary} 

\newtheorem{Lemma}[Theorem]{\quad Lemma} 

\newtheorem{Example}[Theorem]{\quad Example} 

\centerline{}

\begin{abstract} This paper shows that groups of order $64$ are uniquely determined up to isomorphism by their Tables of Marks. This then resolves a previously posed question about whether all groups of order less than $96$ are determined by their Tables of Marks.

\end{abstract} 



\section{Introduction}
In his study of finite groups William Burnside introduced a group invariant called the Table of Marks. \cite{burnside1955theory}\\
It was then found that non-isomorphic groups can have the same Table of Marks. \cite{theevenaz1988isomorphic}.\\
A particularly small example of two groups with isomorphic Tables of Marks was given by the two groups of order $96$ constructed by Raggi-C\'ardens and Valero-Elizondo in 2009. \cite{raggi2009two}\\
In a subsequent paper they raised the question whether these groups are minimal, that is, whether there are any non-isomorphic groups of order strictly less than $96$ with isomorphic Tables of Marks. \cite{martinez2014minimal}\\
Until now it had only been shown, that if two such groups exist, then they must have $64$ elements each.\cite{maldonado2016groups}.\\
In this paper we will solve the remaining case, and show that groups of order $64$ with isomorphic Tables of Marks are already isomorphic groups.\\
Combined with the other papers just mentioned, it then follows that all groups of order less than $96$ are determined by their Tables of Marks. The non-isomorphic groups of order $96$ with isomorphic Tables of Marks are minimal. This answers the question from \cite{martinez2014minimal}.\\

In section \ref{ToMSection} the definition of the Table of Marks is given.\\
In section \ref{MultisetSection} we use multisets to define certain permutation invariants for matrices. We define the multiset of rows and multiset of columns of a matrix, and show that isomorphic Tables of Marks have the same multisets of rows and columns.\\
In section \ref{24section} we distinguish the Tables of Marks of all but $5$ pairs of groups of order $64$ simply by counting how often each integer $n$ appears as an entry in the Tables of Marks.\\
In section \ref{FinalSection} we use the more subtle invariants from section \ref{MultisetSection} to also tell the remaining $5$ pairs apart.\\

Throughout this paper we use the computer algebra system GAP \cite{gap2008gap}. All the GAP programs we wrote for this paper can be found online at \cite{pgap64}.\\

{\bf Acknowledgements.} 
The results of this paper were discovered as part of my Bachelorthesis at the university of Regensburg under the supervision of Niko Naumann.

\section{The Table of Marks}\label{ToMSection}

\begin{Definition}
	Let $G$ be a finite group.\\
	Let $U_1,...,U_n$ be representatives of the conjugacy classes of subgroups of $G$, with $|U_1|\leq |U_2| \leq ... \leq |U_n|$\\
	Write $\mathrm{Fix}_{U_i}(G/U_j) := \{x \in G/U_j | \forall h \in U_i : h \cdot x = x \}$ for the fixpoint set of $U_i$ on $G/U_j$.\\
	Then the matrix $\mathrm{ToM}(G) := (|\mathrm{Fix}_{U_i}(G/U_j)|)_{i,j=1,...,n}$ is called the \underline{\smash{Table of Marks}} of $G$.\\
	Let $\mathscr C(G)$ denote the set of conjugacy classes of subgroups of $G$ and let $Q$ be another finite group. A function $\phi : \mathscr C(G) \rightarrow \mathscr C(Q)$ is called an\\
	\underline{\smash{isomorphism of Tables of Marks from $\mathrm{ToM}(G)$ to $\mathrm{ToM}(Q)$}} if $\phi$ is bijective and for all $i,j$ and for all $V_i \in \phi([U_i]), V_j \in \phi([U_j])$ the equality $|\mathrm{Fix}_{U_i}(G/U_j)| = |\mathrm{Fix}_{V_i}(Q/V_j)|$ holds.
\end{Definition}
The Table of Marks was first introduced in \cite{solomon1967burnside} based on ideas from \cite{burnside1955theory}.\\
Note that the matrix $\mathrm{ToM}(G)$ depends not only on $G$ but also on the ordering of the $U_i$. But for any given group $G$ and any two orderings of the conjugacy classes of subgroups of $G$ the resulting two Tables of Marks of $G$ will be isomorphic Tables of Marks.\\

We will make use of the following obvious fact:\\
If $A = (a_{i,j})_{i,j=1}^n$ and $B = (b_{i,j})_{i,j=1}^n$ are two matrices which are the Tables of Marks of two groups $A = ToM(G), B = ToM(H)$, then $A, B$ are isomorphic Tables of Marks if and only if there is a permutation $\pi \in S_n$ such that $a_{i,j} = b_{\pi(i),\pi(j)}$ for all $i,j$.\\

\section{Multiset of Rows and Columns} \label{MultisetSection}
A convenient language to define objects invariant under permutations is the language of multisets.\\
An overview of the theory of multisets can be found at \cite{blizard1988multiset}.\\
In this section we use multisets to define two invariants of matrices under permutations of rows and columns.
\begin{Definition} 
	A \underline{\smash{multiset}} is a pair $(S,m)$ consisting of a set $S$ called the \underline{\smash{underlying set}} and a function $m : S \rightarrow \mathbb{N}^{\geq 1}$ called the \underline{\smash{multiplicity function}}.\\
	We say that a multiset $(S,m)$ has entries in a set $X$ if $S \subseteq X$.\\
	$PM(X)$ will denote the set of all multisets with entries in $X$.\\
	$PM(X) = \{(S,m)|S \subseteq X $ is a subset $ , m : S \rightarrow \mathbb{N}^{\geq 1} $ is a function$\}$\\
	If $a_1,...,a_n$ are elements of a set $X$ then $[a_1,...,a_n] \in PM(X)$ will denote the multiset, whose underlying set is $\{a_1,...,a_n\}$ and where the multiplicity is defined by $m(a_i) = |\{j \in \{1,...,n\} |a_i = a_j\}|$.\\
	If $T = (a_1,...,a_n)$ is an $n$-tuple of elements from a set $X$, then we also write $\mathrm{MS}(T)$ for $[a_1,...,a_n]$.\\
	
	If $A = (a_{i,j})_{i,j = 1}^n$ is a matrix with entries in a field $k$, and $R^i := (a_{i,1},...,a_{i,n})$ is the $i$-th row of $A$, then we define the \underline{\smash{multiset of rows}} of $A$ to be $\mathrm{Rows}(A) := [\mathrm{MS}(R^1),...,\mathrm{MS}(R^n)] \in PM(PM(k))$, i.e. it is the multiset of all rows of $A$ where the rows are themselves regarded as multisets.
	Similarly if $C^i := (a_{1,i},...,a_{n,i})$ denotes the $i$-th column, we define the \underline{\smash{multiset of columns}}\\ $\mathrm{Columns}(A) := [\mathrm{MS}(C^1),...,\mathrm{MS}(C^n)]$.\\
	
	For the matrix $A$ we also define the \underline{\smash{multiset of entries}} $\mathrm{Entries}(A)$ to be the multiset whose underlying set is $\{a_{i,j}|i,j\in\{1,...,n\} \}$, where the multiplicity is defined by $m(a_{i,j}) = |\{k,l\in \{1,..,n\}|a_{i,j} = a_{k,l} \}|$ 
\end{Definition}
Note that applying a permutation $\pi \in S_n$ to a multiset $[a_1,...,a_n]$ does not change it. $[a_1,...,a_n] = [a_{\pi(1)},...,a_{\pi(n)}]$.\\

In our GAP code we implemented multisets in the following way:\\
A multiset is a pair "[elements, multiplicities]" consisting of two lists "elements" and "multiplicities". The first list "elements" can contain any kind of object of GAP, and is supposed to represent the underlying set. The second list "multiplicities" only contains integers. Both lists have the same length.\\
The $i$-th entry of the second list is supposed to represent the multiplicity of the $i$-th entry of the first list.\\
To compare the equality of two multisets [elements1, multiplicites1] and [elements2, multiplicities2] we check whether the lists "elements1" and "elements2" have the same length, and whether for each object x in "elements1" x is also contained in "elements2" and the corresponding multiplicities match.\\

\begin{Example}
\textnormal{
The matrix
$A = \begin{pmatrix}
1 & 2\\
1 & 2
\end{pmatrix}$
has the same multiset of rows as
$B = \begin{pmatrix}
2 & 1\\
1 & 2
\end{pmatrix}$
 because both matrices have $2$ rows containing exactly one $1$ and one $2$ and both have no other rows.\\
 $\mathrm{Rows}(A) = [[1,2][1,2]] = [[2,1][1,2]] = \mathrm{Rows}(B)$\\
 However the matrices have different multisets of columns, because the first matrix has a column containing two $1$s, while the second matrix has no such column.\\
 $\mathrm{Columns}(A) = [[1,1][2,2]] \neq [[2,1][1,2]] = \mathrm{Columns}(B)$\\
 The two matrices also have the same multisets of entries, because they both contain two 1s and two 2s.\\
 $\mathrm{Entries}(A) = [1,2,1,2] = [2,1,1,2] = \mathrm{Entries}(B)$
}
\end{Example}

We now show that permutations of rows and columns leaves the multiset of rows and columns invariant.\\
\begin{Lemma} \label{MsetRowInvariantLemma}
	Let $A = (a_{i,j})_{i,j=1}^n$ be a matrix and $\pi, \sigma \in S_n$ two permutations. Let $B = (a_{\pi(i),\sigma(j)})_{i,j=1}^n$ be the matrix obtained by applying $\pi$ to the rows and $\sigma$ to the columns of $A$.\\
	Then $\mathrm{Rows}(A) = \mathrm{Rows}(B)$ and $\mathrm{Columns}(A) = \mathrm{Columns}(B)$.
\end{Lemma}
\begin{proof}
	Let $C = (a_{i,\sigma(j)})_{i,j=1}^n$ be the matrix obtained by only applying $\sigma$ to the columns of $A$.\\
	Show that $\mathrm{Rows}(A) = \mathrm{Rows}(C) = \mathrm{Rows}(B)$.\\
	Take $i \in \{1,...,n\}$. Let $R_A^i$ be the $i$-th row of $A$ and $R_C^i$ be the $i$-th row of $C$.\\
	Then $\mathrm{MS}(R_A^i) = [a_{i,1},...,a_{i,n}] \overset{\text{apply }\sigma}{=} [a_{i,\sigma(1)},...,a_{i,\sigma(n)}] = \mathrm{MS}(R_C^i)$ so\\
	$\mathrm{Rows}(A) = [\mathrm{MS}(R_A^1),...,\mathrm{MS}(R_A^n)] = [\mathrm{MS}(R_C^1),...,\mathrm{MS}(R_C^n)] = \mathrm{Rows}(C)$.\\
	If furthermore $R_B^i$ denotes the $i$-th row of $B$, then we have an equality of ordered tuples $R_B^i = R_C^{\pi(i)}$. In particular $\mathrm{MS}(R_B^i) = \mathrm{MS}(R_C^{\pi(i)})$.\\
	Then $\mathrm{Rows}(C) = [\mathrm{MS}(R_C^1),...,\mathrm{MS}(R_C^n)] \overset{\text{apply } \pi}{=} [\mathrm{MS}(R_C^{\pi(1)}),..., \mathrm{MS}(R_C^{\pi(n)})] = \newline [\mathrm{MS}(R_B^1),...,\mathrm{MS}(R_B^n)] = \mathrm{Rows}(B)$.\\
	So $\mathrm{Rows}(A) = \mathrm{Rows}(C) = \mathrm{Rows}(B)$.\\
	$\mathrm{Columns}(A) = \mathrm{Columns}(B)$ follows similarly.
\end{proof}

So if $G, H$ are groups, and $\mathrm{ToM}(G)$ and $\mathrm{ToM}(H)$ are isomorphic Tables of Marks, then $\mathrm{Rows}(\mathrm{ToM}(G)) = \mathrm{Rows}(\mathrm{ToM}(H))$ and $\mathrm{Columns}(\mathrm{ToM}(G)) = \mathrm{Columns}(\mathrm{ToM}(H))$.

We will use these invariants in section \ref{FinalSection}.

Also it is clear that isomorphic Tables of Marks have the same multiset of entries. We will use that fact in the next section.\\

\section{Reducing to 5 pairs} \label{24section}
In this section we will calculate for each group of order $64$ the multiset of entries of its Table of Marks.\\
We will find that there are only $5$ pairs of distinct groups of order $64$ whose Tables of Marks have the same multiset of entries.\\
In particular all other pairs of groups of order $64$ have non-isomorphic Tables of Marks.\\

Most of the work in this section will be done by GAP programs.\\

GAP provides a list of all groups of order $64$ called "AllSmallGroups(64)". We will from now on refer to groups of order $64$ by their position in that list.
So when we for example talk about "group $15$" then that refers to the group that GAP returns when you type in "AllSmallGroups(64)[15]".\\
If $G$ is a group of order $64$ in GAP, then you can find out which number is associated to it by typing into GAP "Position(AllSmallGroups(64),$G$);"\\

In reference \cite{pgap64} you can find a program called "PrintEntryTable.g".\\
After loading that program into GAP you can call the function "printEntriesLatex(64)". It then will calculate for each group of order $64$ the multiset of entries of its Table of Marks and then output the latex code of a table showing that data.\\
The table it outputs was not included in this paper because it was too long, but you can find it in "entryTable.pdf" at \cite{pgap64}, or contact the author.\\

You can then look which pairs of groups have the exact same entries in the table.\\
Since manually comparing all pairs of the $267$ groups of order $64$ is too tedious, we additionally wrote a program called "find5Pairs.g" at \cite{pgap64} that does that for us.

If you load that program into GAP and call the function "findEqualToms(64)" then the program will search for all pairs of distinct groups of order $64$ which have the same multiset of entries. It will then return a list of all such pairs of groups.\\

The only pairs of groups whose Tables of Marks have the same multisets of entries are the groups with the numbers\\
$15$ and $16$\\
$47$ and $48$\\
$106$ and $107$\\
$179$ and $181$\\
$236$ and $240$\\

All other pairs of groups have Tables of Marks with different multisets of entries, which means that the Tables of Marks are not isomorphic.\\

So to show that groups of order $64$ are determined by their Tables of Marks we now just need to show, that these $5$ pairs of groups have non-isomorphic Tables of Marks.

\section{The remaining 5 pairs} \label{FinalSection}

Section \ref{24section} proves that all groups of order $64$, except possibly those with group numbers $15, 16, 47, 48, 106, 107, 179, 181, 236$ and $240$, are uniquely determined by their Tables of Marks.\\
We now just need to show that these five pairs of groups have non-isomorphic Tables of Marks. To do this we will calculate the multisets of rows and columns of their Tables of Marks. By the results of Section \ref{MultisetSection}, if the multiset of rows or the multiset of columns of two Tables of Marks differ, then these Tables of Marks are not isomorphic.\\

The tables  below show the multiset of columns of the remaining groups.\\
Each table shows the multiset of columns of two groups, and allows to compare them easily.\\
In each table, in each line, the left side of the doubleline $\parallel$ specifies a column-multiset, and the right side tells you how many times that type of column occurs in the Tables of Marks of the two groups.
A column-multiset is specified by the amount of $1$s, $2$s, $4$s, $8$s, $16$s, $32$s and $64$s that appear in it. The Table of Marks of the groups in question contain no other non-zero entries.\\
If two numbers right to the doubleline differ, then we underline them.\\
If there is one line in which the two numbers to the right of the doubleline differ, then the two groups have Tables of Marks with different multisets of columns.\\

For example you can read from the first Table below, that there are $3$ columns of the Table of Marks of group $15$ that contain exactly one $16$, two $8$s, three $4$s, three $2$s and one $1$, but that there is only one such column in the Table of Marks of group $16$. This demonstrates that the Tables of Marks of group $15$ and $16$ are not isomorphic.\\

These tables were printed out by the GAP program called "PrintMultisetOfColumnsTable.g" using a method called "PrintMultisetOfColumnsTable()". (\cite{pgap64})\\
The method takes in a pair of two groups, and a boolean value called "printColumnsNotRows" telling it whether to calculate row or column multisets.\\
For example the first table below was printed by passing the method a pair consisting of group 15 and group 16, and by setting the boolean value to true.

\small
\begin{longtable}{c|c|c|c|c|c|c||c|c}
	\#1 & \#2 & \#4 & \#8 & \#16 & \#32 & \#64 & Columns in $G15$ & Columns in $G16$\\
	1 &3 &5 &8 &6 &3 &1 &1 &1\\
	1 &3 &5 &5 &3 &1 &0 &1 &1\\
	1 &3 &5 &6 &3 &1 &0 &1 &1\\
	1 &3 &5 &3 &2 &1 &0 &1 &1\\
	1 &3 &5 &3 &1 &0 &0 &1 &1\\
	1 &3 &3 &3 &1 &0 &0 &\underline{1} & \underline{2}\\
	1 &3 &3 &2 &1 &0 &0 &\underline{3} & \underline{1}\\
	1 &3 &1 &2 &0 &0 &0 &1 &1\\
	1 &3 &3 &1 &0 &0 &0 &2 &2\\
	1 &3 &1 &1 &0 &0 &0 &1 &1\\
	1 &3 &0 &0 &0 &0 &0 &2 &2\\
	1 &1 &2 &0 &0 &0 &0 &1 &1\\
	1 &1 &1 &1 &0 &0 &0 &2 &2\\
	1 &3 &1 &0 &0 &0 &0 &1 &1\\
	1 &1 &1 &0 &0 &0 &0 &2 &2\\
	1 &2 &0 &0 &0 &0 &0 &2 &2\\
	1 &1 &0 &0 &0 &0 &0 &3 &3\\
	1 &0 &0 &0 &0 &0 &0 &1 &1\\
	1 &3 &3 &1 &1 &0 &0 &\underline{0} & \underline{1}\\
	
\end{longtable}
\begin{longtable}{c|c|c|c|c|c|c||c|c}
	\#1 & \#2 & \#4 & \#8 & \#16 & \#32 & \#64 & Columns in $G47$ & Columns in $G48$\\
	1 &3 &5 &5 &5 &3 &1 &1 &1\\
	1 &3 &5 &5 &3 &1 &0 &1 &1\\
	1 &3 &3 &3 &1 &1 &0 &\underline{1} & \underline{0}\\
	1 &3 &3 &3 &3 &1 &0 &\underline{1} & \underline{0}\\
	1 &3 &3 &1 &1 &0 &0 &1 &1\\
	1 &3 &3 &3 &1 &0 &0 &1 &1\\
	1 &3 &5 &3 &1 &0 &0 &1 &1\\
	1 &4 &0 &0 &0 &0 &0 &2 &2\\
	1 &3 &3 &1 &0 &0 &0 &2 &2\\
	1 &3 &0 &0 &0 &0 &0 &2 &2\\
	1 &3 &1 &1 &0 &0 &0 &1 &1\\
	1 &2 &0 &0 &0 &0 &0 &2 &2\\
	1 &3 &1 &0 &0 &0 &0 &1 &1\\
	1 &1 &1 &0 &0 &0 &0 &2 &2\\
	1 &1 &0 &0 &0 &0 &0 &3 &3\\
	1 &0 &0 &0 &0 &0 &0 &1 &1\\
	1 &3 &3 &3 &2 &1 &0 &\underline{0} & \underline{2}\\

\end{longtable}
\begin{longtable}{c|c|c|c|c|c|c||c|c}
	\#1 & \#2 & \#4 & \#8 & \#16 & \#32 & \#64 & Columns in $G106$ & Columns in $G107$\\
1 &7 &15 &19 &15 &7 &1 &1 &1\\
1 &7 &15 &15 &7 &1 &0 &1 &1\\
1 &3 &5 &5 &3 &1 &0 &4 &4\\
1 &7 &11 &9 &5 &1 &0 &\underline{2} & \underline{0}\\
1 &3 &5 &3 &1 &0 &0 &4 &4\\
1 &7 &11 &7 &1 &0 &0 &2 &2\\
1 &3 &3 &2 &1 &0 &0 &\underline{4} & \underline{0}\\
1 &7 &3 &0 &0 &0 &0 &4 &4\\
1 &7 &7 &3 &1 &0 &0 &1 &1\\
1 &7 &7 &1 &0 &0 &0 &1 &1\\
1 &5 &2 &0 &0 &0 &0 &4 &4\\
1 &3 &3 &1 &0 &0 &0 &8 &8\\
1 &3 &1 &1 &0 &0 &0 &2 &2\\
1 &3 &0 &0 &0 &0 &0 &4 &4\\
1 &3 &1 &0 &0 &0 &0 &7 &7\\
1 &1 &1 &0 &0 &0 &0 &4 &4\\
1 &2 &0 &0 &0 &0 &0 &4 &4\\
1 &1 &0 &0 &0 &0 &0 &7 &7\\
1 &0 &0 &0 &0 &0 &0 &1 &1\\
1 &7 &11 &11 &7 &1 &0 &\underline{0} & \underline{1}\\
1 &7 &11 &7 &3 &1 &0 &\underline{0} & \underline{1}\\
1 &3 &3 &1 &1 &0 &0 &\underline{0} & \underline{2}\\
1 &3 &3 &3 &1 &0 &0 &\underline{0} & \underline{2}\\

\end{longtable}
\begin{longtable}{c|c|c|c|c|c|c||c|c}
	\#1 & \#2 & \#4 & \#8 & \#16 & \#32 & \#64 & Columns in $G179$ & Columns in $G181$\\
1 &7 &11 &15 &11 &3 &1 &1 &1\\
1 &7 &11 &11 &3 &1 &0 &1 &1\\
1 &7 &11 &9 &5 &1 &0 &\underline{2} & \underline{0}\\
1 &7 &7 &3 &1 &0 &0 &2 &2\\
1 &7 &11 &7 &1 &0 &0 &1 &1\\
1 &7 &3 &0 &0 &0 &0 &4 &4\\
1 &3 &3 &2 &1 &0 &0 &\underline{4} & \underline{0}\\
1 &3 &3 &1 &0 &0 &0 &2 &2\\
1 &7 &7 &1 &0 &0 &0 &1 &1\\
1 &5 &2 &0 &0 &0 &0 &4 &4\\
1 &3 &1 &1 &0 &0 &0 &4 &4\\
1 &3 &0 &0 &0 &0 &0 &4 &4\\
1 &3 &1 &0 &0 &0 &0 &7 &7\\
1 &2 &0 &0 &0 &0 &0 &4 &4\\
1 &1 &0 &0 &0 &0 &0 &7 &7\\
1 &0 &0 &0 &0 &0 &0 &1 &1\\
1 &7 &11 &11 &7 &1 &0 &\underline{0} & \underline{1}\\
1 &7 &11 &7 &3 &1 &0 &\underline{0} & \underline{1}\\
1 &3 &3 &1 &1 &0 &0 &\underline{0} & \underline{2}\\
1 &3 &3 &3 &1 &0 &0 &\underline{0} & \underline{2}\\

\end{longtable}

For the last pair the multisets of columns happen to agree:
\begin{longtable}{c|c|c|c|c|c|c||c|c}
	\#1 & \#2 & \#4 & \#8 & \#16 & \#32 & \#64 & Columns in $G236$ & Columns in $G240$\\
	1 &15 &35 &33 &26 &7 &1 &1 &1\\
	1 &15 &35 &21 &9 &1 &0 &2 &2\\
	1 &15 &35 &21 &10 &1 &0 &1 &1\\
	1 &7 &13 &5 &0 &0 &0 &4 &4\\
	1 &7 &9 &2 &0 &0 &0 &21 &21\\
	1 &7 &7 &4 &1 &0 &0 &4 &4\\
	1 &15 &35 &15 &1 &0 &0 &1 &1\\
	1 &7 &7 &1 &0 &0 &0 &15 &15\\
	1 &3 &2 &0 &0 &0 &0 &12 &12\\
	1 &4 &1 &0 &0 &0 &0 &6 &6\\
	1 &3 &1 &0 &0 &0 &0 &35 &35\\
	1 &1 &0 &0 &0 &0 &0 &15 &15\\
	1 &0 &0 &0 &0 &0 &0 &1 &1\\
	
\end{longtable}

\normalsize

The only pair of groups whose multisets of columns are equal are the groups $236$ and $240$. So to prove that groups of order $64$ are determined by their Tables of Marks, the only thing left to do, is to show that group $236$ and group $240$ have non-isomorphic Tables of Marks.\\
We now calculate the multiset of rows of these two groups and will see that they are different multisets.\\

\small
\begin{longtable}{c|c|c|c|c|c|c||c|c}
	\#1 & \#2 & \#4 & \#8 & \#16 & \#32 & \#64 & Rows in $G236$ & Rows in $G240$\\
	0 &0 &0 &0 &0 &0 &1 &1 &1\\
	0 &0 &0 &0 &0 &2 &0 &3 &3\\
	0 &0 &0 &1 &0 &1 &0 &4 &4\\
	0 &0 &0 &2 &2 &0 &0 &12 &12\\
	0 &0 &0 &0 &3 &0 &0 &4 &4\\
	0 &0 &0 &1 &2 &0 &0 &9 &9\\
	0 &0 &0 &0 &5 &0 &0 &1 &1\\
	0 &0 &0 &7 &0 &0 &0 &9 &9\\
	0 &0 &4 &3 &0 &0 &0 &\underline{8} & \underline{0}\\
	0 &1 &5 &2 &0 &0 &0 &\underline{4} & \underline{2}\\
	0 &0 &0 &10 &0 &0 &0 &4 &4\\
	0 &1 &3 &2 &0 &0 &0 &\underline{2} & \underline{0}\\
	0 &0 &0 &8 &0 &0 &0 &2 &2\\
	0 &0 &5 &3 &0 &0 &0 &\underline{4} & \underline{8}\\
	0 &0 &18 &0 &0 &0 &0 &\underline{4} & \underline{0}\\
	0 &0 &14 &0 &0 &0 &0 &1 &1\\
	0 &0 &15 &0 &0 &0 &0 &\underline{12} & \underline{14}\\
	0 &0 &13 &0 &0 &0 &0 &\underline{8} & \underline{6}\\
	0 &0 &12 &0 &0 &0 &0 &4 &4\\
	0 &0 &19 &0 &0 &0 &0 &\underline{4} & \underline{2}\\
	0 &0 &21 &0 &0 &0 &0 &\underline{2} & \underline{4}\\
	0 &35 &0 &0 &0 &0 &0 &4 &4\\
	0 &33 &0 &0 &0 &0 &0 &\underline{4} & \underline{0}\\
	0 &41 &0 &0 &0 &0 &0 &\underline{2} & \underline{0}\\
	0 &43 &0 &0 &0 &0 &0 &\underline{4} & \underline{0}\\
	0 &31 &0 &0 &0 &0 &0 &\underline{1} & \underline{2}\\
	118 &0 &0 &0 &0 &0 &0 &1 &1\\
	0 &0 &3 &3 &0 &0 &0 &\underline{0} & \underline{4}\\
	0 &1 &4 &2 &0 &0 &0 &\underline{0} & \underline{4}\\
	0 &0 &16 &0 &0 &0 &0 &\underline{0} & \underline{4}\\
	0 &51 &0 &0 &0 &0 &0 &\underline{0} & \underline{1}\\
	0 &38 &0 &0 &0 &0 &0 &\underline{0} & \underline{8}\\
\end{longtable}
\normalsize

As we can see the multiset of rows of the Tables of Marks of the groups $236$ and $240$ are non-isomorphic.\\
And this concludes the proof, that groups of order $64$ are determined by their Table of Marks.

\bibliographystyle{plain}

\newpage

\end{document}